\documentclass[reqno]{amsart}
\usepackage{latexsym}
\usepackage{amssymb,amsthm,amsmath}

\theoremstyle{plain}
\newtheorem{theorem}{Theorem}
\newtheorem{lemma}[theorem]{Lemma}
\newtheorem{corollary}[theorem]{Corollary}

\theoremstyle{definition}
\newtheorem{definition}[theorem]{Definition}
\theoremstyle{remark}
\newtheorem{remark}[theorem]{Remark}

\def\d#1{{#1\kern-0.4em\char"16\kern-0.1em}}
\def\D#1{{\raise0.2ex\hbox{-}\kern-0.4em #1}}
\newcounter{zd}

\newcounter{zdr}[subsection]

\def\R{I\!\!R}

\def\N{I\!\!N}
\def\C{I\!\!\!C}

\def\F{{\cal F}}
\def\pa{\partial}

\def\cal{\mathcal}

\def\dscon{\relbar\joinrel\rightharpoonup}
\def\Dscon{\relbar\joinrel\dscon}
\def\dup#1#2#3#4{{}_{#1\!}\langle\, #2 , #3 \,\rangle_{#4}}
\def\nor#1#2{{\| #1 \|}_{#2}}
\def\povrhsk#1{\smash{
        \mathop{\;\Dscon\;}\limits^{#1}}}
\def\oi#1#2{\langle#1,#2\rangle}
\def\zoi#1#2{[#1,#2\rangle}

\def\supp{{\rm supp\,}}
\def\lG{{\cal G}}
\def\lA{{\cal A}}

\let\ph=\varphi

\begin{document}
\title{ $H$-distributions --- an extension of $H$-measures}

\thanks{The work of N.A.~is supported in part by the Croatian
MZOS through projects 037--0372787--2795 and 037--1193086--3226.}

\author{N.~Antoni\'c}
\address{Nenad Antoni\'c, Department of Mathematics, Faculty of Science, University of Zagreb,
    Bijeni\v{c}ka c.~30,    HR--10000 Zagreb, Croatia }\email{nenad@math.hr}
\author{D.~Mitrovic}\address{Darko Mitrovic, University of Montenegro,
    Faculty of Mathematics, Cetinjski put bb, 81000 Podgorica, Montenegro and }
\address{ University of Bergen, Faculty of
    Mathematics, Johannes Bruns gt. 12, Bergen,
    Norway}\email{matematika@t-com.me}

\date{}



\begin{abstract}
We use the continuity  of Fourier multiplier operators on $L^p$ to introduce
the $H$-distributions --- an extension of $H$-measures in the
$L^p$ framework. We apply the $H$-distributions to obtain an $L^p$
version of the localisation principle, and reprove the $L^p$---$L^q$
variant of the Murat--Tartar div-curl lemma.
\end{abstract}

\maketitle

\section{Introduction}

In the study of partial differential equations, quite often it is of
interest to determine whether some $L^p$ weakly convergent
sequence converges strongly. Various techniques and tools have been
developed for that purpose (for the state of the art twenty years
ago see \cite{Evans}); of more modern ones we only mention the
{\sl H-measures}\/ of Luc Tartar \cite{Tar}, independently
introduced by Patrick G\'erard \cite{Ger} under the name of {\sl
microlocal defect measures}. $H$-measures proved to be very powerful
tool in a number of applications (see e.g. \cite{MA1, MA2, Ant3,
AntV, ?2, ?9, KohnRV, ?3, Pan} and references therein, which is
surely an incomplete list). The main theorem on the existence of
H-measures, in an equivalent form suitable for our purposes, reads:

\begin{theorem}
\label{tbasic1}
If scalar sequences $u_n, v_n \dscon 0$ weakly in
$L^2(\R^d)$, then there exist subsequences $(u_{n'}), (v_{n'})$ and
a complex Radon measure $\mu$ on $\R^d\times S^{d-1}$ such that for
every $\ph_1,\ph_2\in C_0(\R^d)$ and every $\psi\in C(S^{d-1})$
\begin{equation}
\label{dec427}
  \quad
    \lim_{n'} \langle {\cal A}_\psi(\ph_1 u_{n'})| \ph_2 v_{n'} \rangle =
    \lim_{n'} \int_{\R^d} {\cal A}_\psi(\ph_1 u_{n'}) \overline{\ph_2 v_{n'}}\,dx = \langle \mu,\ph_1\bar\ph_2\bar\psi\rangle \;,
\end{equation} where ${\cal A}_\psi$ is the Fourier multiplier operator with
the symbol $\psi$:
$$
{\cal A}_\psi u := \bar{\cal F}(\psi\hat u)\;.
$$
\end{theorem}

The measure $\mu$ we call the {\em $H$-measure} corresponding to the
sequence $(u_n,v_n)$.
In fact, it corresponds to the off-diagonal element of the corresponding $2\times2$
matrix Radon measure of the vector function $(u_n, v_n)$ (cf.~\cite{Ant}).

\begin{remark}
\label{rem_nenad}
After applying the Plancherel theorem, the term under
the limit sign in Theorem \ref{tbasic1} takes the form
\begin{equation}
\label{ns30} \int_{\R^d} \widehat{\ph_1 u_{n'}}
\overline{\widehat{\ph_2 v_{n'}}}\psi\,d\xi\;,
\end{equation}
where by $\hat u(\xi) = ({\cal F}u)(\xi) = \int_{\R^d} e^{-2\pi
ix\cdot\xi} u(x)\,dx$ we denote the Fourier transform on $\R^d$
(with the inverse $(\bar{\cal F}v)(x) := \int_{\R^d} e^{2\pi
ix\cdot\xi} v(\xi)\,d\xi$).

In the particular case of $u_n=v_n$, $\mu$ roughly describes the loss
of strong $L^2$ precompactness of sequence $(u_n)$. Indeed, it
is not difficult to see that if $(u_n)$ is
strongly convergent in $L^2$, then the corresponding H-measure is
trivial; on the other hand, if the H-measure is trivial,
then $u_n\longrightarrow 0$ in $L^2_{loc}(\R^d)$ (for the details in a
similar situation see \cite{Ant4}).
\end{remark}

In order to explain how to apply this idea to $L^p$-weakly
converging sequences when $p\ne 2$, consider the integral in
\eqref{dec427}. The Cauchy-Schwartz inequality and the Plancherel
theorem imply (see e.g. \cite[p.~198]{Tar})
\begin{equation}
\Big| \int_{\R^d}{\cal A}_\psi(\varphi_1 u_{n'})\overline{\varphi_2 v_{n'}}dx\Big|\leq C
\|\psi\|_{C(S^{d-1})}\|\varphi_1 \overline{\varphi_2}\|_{C_0(\R^d)}\;,
\end{equation}
where $C$ depends on a uniform bound for $\|(u_{n},v_{n})\|_{L^2(\R^d;\R^2)}$.
Roughly speaking, this fact and the
linearity of integral in \eqref{dec427} with respect to
$\varphi_1\overline{\varphi_2}$ and $\psi$ enable us to state that
the limit in \eqref{dec427} is a Radon measure (a bounded linear functional on
$C_0(\R^d\times S^{d-1})$). Furthermore, the bound is obtained by a
simple estimate $\|{\cal A}_\psi\|_{L^2\to L^2}\leq
\|\psi\|_{L^\infty(\R^d)}$ and the fact that $(u_n,v_n)$ is a
bounded sequence in $L^2(\R^d;\R^2)$.

In \cite{Ger}, the (natural) question whether it is possible to
extend the notion of $H$-measures (or microlocal defect measures in
the terminology used there) to the $L^p$ framework  is posed (see
also \cite[p. 331.]{tar_book}). We shall consider only the case
$p\in\langle1,\infty\rangle$ (i.e.~$1<p<\infty$), while its dual
exponent we consistently denote by $p'$.

To answer that question, one necessarily needs precise bounds for the Fourier
multiplier operator ${\cal A}_\psi$ as a mapping from $L^p(\R^d)$ to
$L^p(\R^d)$. The bounds are given by the famous H\"ormander-Mikhlin
theorem \cite{Oki, Gra}.

%Let us first recall the definition of Fourier multiplier operators.

\begin{definition}
\label{mult} Let $\phi:\R^d\to \C$ satisfy $(1+|x|^2)^{-k/2}\phi\in
L^1(\R^d)$ for some $k\in \N_0$. Then $\phi$ is called the {\em Fourier
multiplier}\/ on $L^p(\R^d)$, if $\overline\F(\phi
\F(\theta))\in L^p(\R^d)$ for any $\theta\in {\cal S}(\R^d)$, and
$$
{\cal S}(\R^d)\ni\theta\mapsto \overline\F(\phi\F(\theta)) \in L^p(\R^d)
$$
can be extended to a continuous mapping $T_\phi: L^p(\R^d)\to
L^p(\R^d)$. Operator $T_\phi$ we call the {\em $L^p$-multiplier operator}\/
with {\em symbol}\/ $\phi$.
\end{definition}

\begin{theorem}[{\bf H\"ormander-Mikhlin}\/]
\label{tHM1}
Let $\phi\in L^\infty(\R^d)$
have partial derivatives of order less than or equal to $\kappa$,
where $\kappa$ is the least integer strictly greater than $d/2$
(i.e.~$\kappa=[\frac{d}{2}]+1$). If for some constant
$k>0$
\begin{equation}
\label{cond_1'} (\forall r>0)(\forall \alpha\in \N_0^d)\qquad
    |\alpha|\leq \kappa \;\Longrightarrow\;
\int_{\frac{r}{2}\leq \|\xi\|\leq
r}|D_{\xi}^{\alpha}\phi(\xi)|^2d\xi\leq k^2r^{d-2n(\alpha)}\;,
\end{equation}
then for any $p\in\oi1\infty$ and the associated multiplier operator
$T_\phi$ there exists a constant $C_d$ (depending only on the
dimension $d$; see \cite[p. 362]{Gra}) such that
\begin{equation}
\label{*} \|T_\phi\|_{L^p\to L^p}\leq  C_d p
(p-1)(k+\|\phi\|_\infty) \;.
\end{equation}
\end{theorem}

\begin{remark} \label{rem_stein}
It is important to notice that according to \cite[Sect.~3.2, Example
2]{Ste}, if the symbol of a multiplier is a $C^\kappa$ function defined
on the unit sphere $S^{d-1}\subseteq \R^d$, then the constant $k$ from
Theorem \ref{tHM1} can be taken to be equal to $\|\phi\|_{C^\kappa(S^{d-1})}$.
\end{remark}

%   NE ZNAM ČEMU TO SLUŽI, PA IZOSTAVLJAM
%We refer also to papers
%\cite{kw} and \cite{mwy} and the references therein for the  norm
%inequalities for the weighted $L^p$ multipliers.

By an application of Theorem \ref{tHM1}, in Section 2 we are able to
introduce the $H$-distribu\-tions (see Theorem \ref{tbasic2} below)
--- an extension of $H$-measures in the $L^p$-setting. Its proof
is the main result of the paper, and forms Section 3. We conclude Section 4 by an
$L^p$-variant of the localisation principle and a proof of an
$(L^p,L^{p'})$-variant of the div-curl lemma.

%Before we give a formal definition of a multiplier and the
%corresponding multiplier operator, we will introduce notations that
%we shall use.

\begin{remark}
Recently, variants of $H$-measures with a different scaling were
introduced (the parabolic $H$-measures \cite{Ant2, Ant3} and the
ultra-parabolic $H$-measures \cite{HKMP}). We can apply the
procedure from this paper to extend the notion of such $H$-measures
to the $L^p$-setting in the same fashion as it is given here based on Theorem \ref{tbasic1}
for the classical $H$-measures.
\end{remark}

\section{A generalisation of $H$-measures}

We have already seen (Remark \ref{rem_nenad}) that an $H$-measure
$\mu$ corresponding to a sequence $(u_n)$ in $L^2(\R^d)$ can
describe its loss of strong compactness.
We would like to introduce a similar notion describing the loss
(at least in $L^1_{loc}$) of strong compactness for sequences weakly converging in $L^p(\R^d)$.
Our extension is motivated by the following lemma \cite[Lemma 7]{DHM} and its corollary.

\begin{lemma} \cite{DHM}
\label{DHM} For $l\in \R^+$ and $u\in \R$ denote
\begin{equation}
T_l(u)=\begin{cases} l, & u>l\\
u, & u\in [-l,l]\\
-l, & u<-l\;.
\end{cases}
\label{trunc}
\end{equation}
Assume that a sequence $(u_n)$ of measurable functions on
$\Omega\subseteq \R^d$ is such that
\begin{equation}
\label{novo-1} \sup\limits_{n\in \N}\int_{\Omega}|u_n|^p dx <\infty
\;.
\end{equation}
Suppose further that for each fixed $l>0$ the sequence of truncated functions
$(T_l(u_n))_{n}$ is precompact in $L^1(\Omega)$.
Then, there exists a measurable function $u$ such that on a
subsequence
$$
u_{n_k} \to u  \ \ \text{in measure.}
$$
\end{lemma}

\begin{corollary}
\label{DHMcor} The subsequence in Lemma \ref{DHM} satisfies
$$
u_{n_k} \to u  \ \ \text{strongly in $L^1_{loc}(\Omega)$.}
$$
\end{corollary}

\begin{proof}
%   N.B. U Follandu to nisam našao. Ima u Lieb, Loss: Analysis, AMS, 2001.
%   Naravno, piše i u Markovim i mojim skriptama :-)
By the Lieb form of Fatou's lemma we conclude that $u\in L^p(\Omega)$. Furthermore,
for any compact $K\subseteq\Omega$ on the limit $k\to \infty$ we have
\begin{align*}
\int_{K}\!|u_{n_k}-u|dx
& = \int_{\{|u_{n_k}-u|>1/k\}\cap K}\!|u_{n_k}-u|dx
    + \int_{\{|u_{n_k}-u|\leq 1/k\}\cap K}\!|u_{n_k}-u|dx\\
&\leq \left({m}\big(\{|u_{n_k}-u|>1/k\}\cap
K\big)\right)^{p'}\!\int_{K}\!|u_{n_k}-u|^p dx
    + m(K)/k\longrightarrow 0\;.
\end{align*}
\vskip-3mm

\end{proof}

%Another interesting property of the truncation operator which we
%will make use of is given by the following lemma.

%\begin{lemma}
%\label{l-novo} Assume that the sequence $(u_n)$ satisfies
%\eqref{novo-1} with $p>1$. Then, for any $q\in \zoi1p$,
%\begin{equation}
%\label{novo-2} T_l(u_n)-u_n \to 0 \ \ \text{strongly in
%$L^q_{loc}(\R^d)$} ,
%\end{equation}
%as $l\to \infty$, uniformly in $n\in \N$.
%In other words, for any compact $K\subseteq \R^d$
%\begin{equation*}
%\lim\limits_{n\to\infty}\lim\limits_{l\to
%\infty}\int_K|T_l(u_n)-u_n|^q dx
%=\lim\limits_{l\to\infty}\lim\limits_{n\to \infty}\int_K
%|T_l(u_n)-u_n|^q dx=0.
%\end{equation*}%

%\end{lemma}

%\begin{proof}
%Denote by $C$ the $L^p$-bound of the sequence $(u_n)$; for
%a compact $K\subseteq \R^d$
%\begin{align*}
%\int_{K}|T_l(u_n)-u_n|^q dx
%&\leq \int_{\{|u_n|>l\}}|2u_n|^q dx \\
%&\leq \left(\int_{\R^d} |2u_n|^p \right)^{\frac{q}{p}}
%\left(\int_{\{|u_n|>l\}} dx \right)^{\frac{p-q}{p}} \\
%& \leq C \left({\rm meas}\{|u_n|>l \} \right)^{\frac{p-q}{p}} \to 0 ,
%\end{align*}
%%
%as $l\to \infty$, uniformly in $n\in \N$.%

%Indeed, if there exists a function $l\mapsto n(l)\in \N$ such that
%$n(l)\to \infty$ as $l\to \infty$, and
%$$
%{\rm meas}\{u_{n(l)}>l \} \geq c,
%$$for some $c>0$ and every $l>0$, we have
%\begin{align*}
%C > \int_{\{|u_{n(l)}|>l \}}|u_n|^p dx > l^p \cdot {\rm
%meas}\{u_{n(l)}>l\}> l^p \cdot c\to \infty \ \ {\rm as} \ \ l\to
%\infty
%\end{align*}which is an obvious contradiction.
%\end{proof}

From Corollary \ref{DHMcor} we see that if we want to analyse the
strong $L^1_{loc}$ compactness for a sequence $(u_n)$ weakly
converging to zero in $L^p(\R^d)$, it is enough to inspect how the
truncated sequences $(v_{n,l})_{n}:=(T_l(u_n))_{n}$ behave.
Furthermore, notice that it is not enough to consider
$(v_{n,l})_{n,l}$ independently of $(u_n)$ since this would force us
to estimate $u_n-v_{n,l}$, which is usually not easy. For instance,
consider a sequence $(u_n)$ weakly converging to zero in
$L^p(\R^d)$, and solving the following family of problems:
\begin{equation}
\label{ns31} \sum\limits_{i=1}^d
\pa_{x_i}\left(A_i(x)u_n(x)\right)=f_n(x),
\end{equation}
where $A_i\in C_0(\R^d)$ and $f_n\to 0$ strongly in the Sobolev
space $H^{-1}(\R^d)$.
When dealing with the latter equation it is standard to
multiply \eqref{ns31} by ${\cal A}_{\frac{\psi}{|\xi|}}(\phi u_n)$,
for $\phi\in C_0(\R^d)$, where ${\cal A}_{\frac{\psi}{|\xi|}}$ is the
multiplier operator with symbol $\frac{\psi(\xi/|\xi|)}{|\xi|}$,
$\psi\in C(S^{d-1})$, and then pass
to the limit (see e.g.~\cite{Ant, Sazh}). If $u_n\in
L^2(\R^d)$, we can apply the classical $H$-measures to describe the
defect of compactness for $(u_n)$.

If we instead take $u_n\in L^p(\R^d)$, for $p<2$, we can try to
rewrite \eqref{ns31} in the form
\begin{equation*}
\sum\limits_{i=1}^d
\pa_{x_i}\left(A_i(x)T_l(u_n)(x)\right)=f_n(x)+\sum\limits_{i=1}^d
\pa_{x_i}\left(A_i(x)(T_l(u_n)(x)-u_n(x))\right),
\end{equation*} and, similarly as before, to
multiply \eqref{ns31} by ${\cal A}_{\frac{\psi}{|\xi|}}(\phi T_l(u_n))$.
Unfortunately, we are not
able to control the right-hand side of such an expression and we
need to change the strategy. In view of these considerations,
we formulate the following theorem.

\begin{theorem}
\label{tbasic2} If $u_n\dscon 0$ in $L^{p}({\R}^{d})$ and $v_n\dscon
v$ in $L^q(\R^{d})$ for $q \geq p'$, then there exist subsequences
$(u_{n'})$, $(v_{n'})$ and a complex valued distribution $\mu\in
{\cal D}'(\R^{d} \times S^{d-1})$ of order not more than
$\kappa=[d/2]+1$, such that for every $\varphi_1, \varphi_2\in
C_c(\R^{d})$ and $\psi\in C^{\kappa}(S^{d-1})$ we have:
\begin{equation}
\label{basic111}
\begin{split}
\lim\limits_{n'\to \infty}\!\int_{\R^{d}}\!{{\cal
A}_{\psi}(\varphi_1 u_{n'})(x)}\overline{(\varphi_2
v_{n'})(x)}dx&\!=\!\lim\limits_{n'\to
\infty}\!\int_{\R^{d}}\!(\varphi_1 u_{n'})(x)\overline{{\cal
A}_{\overline{\psi}}(\varphi_2 v_{n'})(x)}dx\\&=\langle\mu
,\varphi_1\overline\varphi_2\psi \rangle,
\end{split}
\end{equation}
where ${\cal A}_{\psi}:L^p(\R^d)\to L^p(\R^d)$ is a multiplier
operator with symbol $\psi\in C^\kappa(S^{d-1})$.

\end{theorem}

We call the functional $\mu$ the {\em H-distribution} corresponding
to (a subsequence of) $(u_n)$ and $(v_n)$.

\begin{remark}
Notice that, unlike to what was the case with $H$-measures, it
is not possible to write \eqref{basic111} in a form similar to
\eqref{ns30} since, according to the Hausdorff-Young inequality,
$\|\F(u)\|_{L^{p'}(\R^d)}\leq C \|u\|_{L^{p}(\R^d)}$
only if $1<p<2$. This means that we are not able to estimate
$\|\F(\varphi_2 v_n)\|_{L^{q}(\R^d)}$, $q>2$, which would appear
from \eqref{basic111} when rewriting it in a form similar to
\eqref{ns30}.
\end{remark}

\section{Proof of Theorem \ref{tbasic2}}

In order to prove the theorem, we need a consequence of Tartar's
First commutation lemma \cite[Lemma 1.7]{Tar}. First, for $a\in
C^{\kappa}(S^{d-1})$ and $b\in C_0(\R^d)$ define the Fourier
multiplier operator ${\cal A}_\psi$ and the operator of multiplication
$B$ on $L^p(\R^d)$, by the formulae:
\begin{align}
\label{oper_1}
&\F({\cal A}_\psi u)(\xi)=\psi\Bigl(\frac{\xi}{|\xi|}\Bigr)\F(u)(\xi) \ ,\\
\label{oper_2} &Bu(x)=b(x)u(x) \ .
\end{align}
Notice that $\psi$ satisfies the conditions of the
H\"{o}rmander-Mikhlin theorem and that $\|{\cal A}_{\psi}\|_{L^p\to
L^p}\leq C \|\psi\|_{C^\kappa}$ (see Remark \ref{rem_stein}).
Therefore, ${\cal A}_\psi$ and $B$ are bounded operators on
$L^p(\R^d)$, for any $p\in\oi1\infty$. We are interested in the
properties of their commutator, $C={\cal A}_\psi B-B{\cal A}_\psi$.

\begin{lemma}
\label{scl} Let $(v_n)$ be bounded in both $L^2(\R^d)$ and
$L^\infty(\R^d)$, and such that $v_n\rightharpoonup 0$ in the sense
of distributions. Then the sequence $(Cv_n)$ strongly converges to
zero in $L^q(\R^d)$, for any $q\in\zoi2\infty$.
\end{lemma}

\begin{proof}
First, notice that we do not have the boundedness of ${\cal A}_\psi$ on
$L^\infty$, but only on $L^p$, for $p<\infty$.
Therefore we take $p\in\oi q\infty$, and by the classical interpolation
inequality conclude that $(v_n)$ is bounded in $L^p$.
Now we can apply the same inequality again:
\begin{equation}
\label{pro_1} \|Cv_n\|_{q}\leq
\|Cv_n\|^\alpha_{2}\|Cv_n\|^{1-\alpha}_{p},
\end{equation}
for $\alpha\in\oi01$ such that $1/q=\alpha/2+(1-\alpha)/p$. As $C$ is
a compact operator on $L^2(\R^d)$ by the First commutation lemma,
while $C$ is bounded on $L^p(\R^d)$, from \eqref{pro_1} we get the
claim.
\end{proof}

\noindent{\bf Proof of Theorem \ref{tbasic2}:} The first equality
from \eqref{basic111} follows from the fact that the adjoint
operator ${\cal A}^*_{\psi}$ corresponding to ${\cal A}_\psi$ is
actually the multiplier operator ${\cal A}_{\bar{\psi}}$ (see
\cite[Theorem 7.4.3]{Oki}). This means that (we take the duality
product to be sesquilinear, i.e.~antilinear in the second variable,
in order to get the scalar product when $p=p'=2$)
$$
_{L^p} \langle {\cal A}_\psi(\varphi_1 u_{n'}),\varphi_2 v_{n'}
\rangle_{L^{p'}}=_{L^p}\langle \varphi_1 u_{n'}, {\cal
A}_{\overline{\psi}}(\varphi_2 v_{n'}) \rangle_{L^{p'}},
$$ which is exactly what we need.
We can now concentrate our attention to the second equality in \eqref{basic111}.

Since $u_{n}\rightharpoonup 0$ in $L^{p}(\R^{d})$, while for $v\in
L^\infty(\R^d)$ we have $\varphi_1{{\cal A}_{\psi}(\varphi_2 v)}\in
L^{p'}(\R^{d})$, according to the H\"ormander-Mikhlin theorem for
any $\varphi_1, \varphi_2\in C_c(\R^{d})$ and $\psi\in
C^\kappa(S^{d-1})$, it follows that
\begin{equation}
\label{novo1} \lim\limits_{n\to\infty}\int_{\R^{d}}\varphi_1u_{n}
    \overline{{\cal A}_{\overline\psi}(\varphi_2v)}dx= 0. \nonumber
\end{equation}

We can write $\R^d=\bigcup_{l\in\N} K_l$, where $K_l$ form an
increasing family of compact sets (e.g.~closed balls around the
origin of radius $l$); therefore $\supp\varphi_2\subseteq K_l$ for
some $l\in \N$. We have:
\begin{align*}
\label{*3} \lim\limits_{n\to
\infty}\int_{\R^{d}}\varphi_1u_{n}\overline{{\cal A}_{\overline\psi}(\varphi_2
v_{n})}dx
&= \lim\limits_{n\to \infty}\int_{\R^{d}}\varphi_1u_{n}\overline{{\cal A}_{\overline\psi}[\varphi_2 \chi_{l}(v_{n}-v)]}dx \\
&= \lim\limits_{n\to \infty}\int_{\R^{d}}\varphi_1\overline\varphi_2  u_{n}\overline{{\cal A}_{\overline\psi}(\chi_{l}( v_{n}-v))}dx \\
&= \lim\limits_{n\to \infty}\int_{\R^{d}}\varphi_1\overline\varphi_2 u_{n}\overline{{\cal A}_{\overline\psi}( \chi_{l} v_{n})}dx,
\end{align*}
where $\chi_l$ is the characteristic function of $K_l$. In the
second equality we have used Lemma \ref{scl}.

This allows us to express the above integrals as bilinear
functionals, after denoting $\varphi= \varphi_1\overline\varphi_2$:
\begin{equation}
\label{*4} \mu_{n,l}(\varphi,\psi)\!=\! \int_{\R^{d}}\!\!\varphi
u_{n}\overline{{\cal A}_{\overline\psi}(\chi_{l} v_{n})}dx.
\end{equation}
Furthermore, $\mu_{n,l}$ is bounded by
$\tilde{C}\|\varphi\|_{C_0(\R^{d})} \|{\psi}\|_{C^\kappa(S^{d-1})}$,
as according to the H\"{o}lder inequality and Remark
\ref{rem_stein}:
\begin{align*}
\label{*09} \big|   \mu_{n,l}(\varphi,\psi) \big|
 \leq \|\varphi u_{n}\|_{p}\|{\cal A}_\psi(\chi_{l} v_{n})\|_{p'}
&\leq \tilde{C} \| \psi\|_{C^\kappa(S^{d-1})}
\|\varphi\|_{C_0(\R^{d})},
\end{align*}
where the constant $\tilde{C}$ depends on $L^p(K_l)$-norm and
$L^{p'}(K_l)$-norm of the sequences $(u_n)$ and $(v_n)$,
respectively.

For each $l\in\N$ we can apply Lemma \ref{bil} below to obtain
operators $B^l\in {\cal L}(C_{K_l}(\R^d); (C^\kappa(S^{d-1}))')$.
Furthermore, for the construction of $B^l$, we can  start with a
defining subsequence for $B^{l-1}$, so that the convergence will
remain valid on $C_{K_{l-1}}(\R^d)$, in such a way obtaining that
$B^l$ is an extension of $B^{l-1}$.

This allows us to define the operator $B$ on $C_c(\R^d)$: for
$\varphi\in C_c(\R^d)$ we take $l\in\N$ such that
$\supp\varphi\subseteq K_l$, and set $B\varphi:=B^l\varphi$. Because
of the above mentioned extension property, this definition is good,
and we have a bounded operator:
$$
\|B\varphi\|_{(C^\kappa(S^{d-1}))'} \le \tilde C
\|\varphi\|_{C_0(\R^d)} \;.
$$

In such a way we got a bounded linear operator $B$ on the space
$C_c(\R^d)$ equipped with the uniform norm; the operator can be
extended to its completion, the Banach space $C_0(\R^d)$.

Now we can define $\mu(\varphi,\psi):= \langle
B\varphi,\psi\rangle$, which satisfies \eqref{basic111}.

We can restrict $B$ to an operator $\tilde B$ defined only on
$C^\infty_c(\R^d)$; as the topology on $C^\infty_c(\R^d)$ is
stronger than the one inherited from $C_0(\R^{d})$, the restriction
remains continuous. Furthermore, $(C^\kappa(S^{d-1}))'$ is the space
of distributions of order $\kappa$, which is a subspace of ${\cal
D}'(S^{d-1})$. In such a way we have a continuous operator from
$C^\infty_c(\R^d)$ to ${\cal D}'(S^{d-1})$, which by the Schwartz
kernel theorem can be identified to a distribution from ${\cal
D}'(\R^d\times S^{d-1})$ (for details cf.~\cite[Ch.~VI]{HormB}).
\hfill $\Box$
\medskip

We conclude this section by a simple lemma and its proof, which was
used in the proof of Theorem \ref{tbasic2}.

\begin{lemma}
\label{bil} Let $E$ and $F$ be separable Banach spaces, and $(b_n)$
an equibounded sequence of bilinear forms on $E\times F$ (more
precisely, there is a constant $C$ such that for each $n\in\N$ we
have $|b_n(\ph, \psi)| \le C\|\ph\|_E\|\psi\|_F$).

Then there exists a subsequence $(b_{n_k})$ and a bilinear form $b$
(with the same bound $C$) such that
$$
(\forall\ph\in E)(\forall\psi\in F)\qquad \lim_k b_{n_k}(\ph,\psi) =
b(\ph,\psi)\;.
$$
\end{lemma}

\begin{proof}
To each $b_n$ we associate a bounded linear operator $B_n: E
\longrightarrow F'$ by
$$
\dup{F'}{B_n\ph}\psi F := b_n(\ph,\psi) \;.
$$
The above expression clearly defines a function (i.e.~$B_n\ph\in F'$
is uniquely determined), it is linear in $\ph$, and bounded:
$$
\nor{B_n\ph}{F'} = \sup_{\psi\ne0}{|b_n(\ph,\psi)|\over \nor\psi F}
\le C\nor\ph E \;.
$$
Let $\lG\subseteq E$ be a countable dense subset; for each
$\ph\in\lG$ the sequence $(B_n\ph)$ is bounded in $F'$, so by the
Banach-Alaoglu-Bourbaki theorem there is a subsequence such that
$$
B_{n_1}\ph \povrhsk\ast \beta_1 =: B(\ph) \;.
$$
By repeating this construction countably many times, and then
applying the Cantor diagonal procedure we get a subsequence
$$
(\forall \ph\in\lG)\qquad B_{n_k}\ph \povrhsk\ast B(\ph) \;,
$$
such that $\nor{B(\ph)}{F'} \le C\nor\ph E$.

Then it is  standard to extend $B$ to a bounded linear operator on
the whole space $E$. Clearly:
$$
b(\ph,\psi):=\dup{F'}{B\ph}\psi F = \lim_k \dup{F'}{B_{n_k}\ph}\psi
F = \lim_k b_{n_k}(\ph,\psi)\;.
$$
\end{proof}

\section{Some applications}

In applications quite often it is needed to prove that a weakly convergent
sequence is, at the same time, strongly convergent (see e.g.
\cite{MA1, chen, Pan, Sazh}). In view of Corollary \ref{DHMcor}, in
order to prove strong $L^1_{loc}$ convergence of a weakly
convergent sequence, the given version of Theorem \ref{tbasic2} is
sufficient. Indeed, assume that $u_n \rightharpoonup0$ in
$L^p(\R^d)$. Denote $v^l_n=T_l(u_n)$ and assume that we are able to
prove that the $H$-distribution $\mu^l$ corresponding to
subsequences $(u_{n'})$ and $(v^l_{n'})$ is identically equal to
zero for each $l\in\N$. In that case, taking $\psi=1$,
$\varphi_1=\varphi_2=\varphi$ in \eqref{basic111}, we have:
\begin{align*}
0=\lim\limits_{n'\to \infty}\int_{\R^{d}}\varphi u_{n'}{{\cal
A}_{1}(\varphi v^l_{n'})}dx
& =  \lim\limits_{n'\to \infty}\int_{\R^{d}}\varphi^2 u_{n'} T_l(u_{n'})dx \\
& \geq   \lim\limits_{n'\to \infty}\int_{\R^{d}}\varphi^2
|T_l(u_{n'})|^2dx\;.
\end{align*}
This implies that for any fixed $l\in\N$ we have $v^l_{n'}\longrightarrow 0$ strongly $L^2_{loc}$,
implying the same convergence in $L^1_{loc}$. Now by Corollary
\ref{DHMcor} we conclude that $u_n\longrightarrow 0$ in $L^1_{loc}$.
Comparing the latter to Remark \ref{rem_nenad}, we see that
$H$-distributions are a proper generalisation of $H$-measures.
Actually, the following localisation principle holds (see also
\cite[Theorem 1.6]{Tar} and \cite[Theorem 2]{Ant}).

\begin{theorem}\label{th_loc}
Consider \eqref{ns31}, under the assumptions that
$u_n\rightharpoonup 0$ in $L^p(\R^d)$, and $f_n\to0$ in
$W^{-1,q}(\R^d)$, for some $q\in\oi1d$. Take an arbitrary sequence
$(v_n)$ bounded in $L^\infty(\R^d)$, and by $\mu$ denote the
$H$-distribution corresponding to some subsequences of sequences
$(u_n)$ and $(v_n)$. Then
\begin{equation}
\label{pro_7} \sum\limits_{i=1}^dA_i(x)\xi_i \mu(x,\xi)=0 \ ,
\end{equation}
in the sense of distributions on $\R^d\times S^{d-1}$,
the function $(x,\xi)\mapsto \sum\limits_{i=1}^dA_i(x)\xi_i$ being the symbol
of the linear partial differential operator with $C^\kappa_0$ coefficients.
\end{theorem}

\begin{proof}
In order to prove the theorem, we need a particular multiplier, the
so called (Marcel) Riesz potential $I_1:=\lA_{|2\pi\xi|^{-1}}$, and
the Riesz transforms $R_j:=\lA_{\xi_j\over i|\xi|}$
\cite[V.1,2]{Ste}.  We note that [id.,V.2.3]
\begin{equation}\label{V23}
\int I_1(\phi)\partial_{x_j} g = \int (R_j \phi)g, \ \ g\in {\cal
S}(\R^d).
\end{equation}
From here, using the density argument and the fact that $R_j$ is bounded from
$L^p(\R^d)$ to itself, we conclude that $\pa_j I_1(\phi)=-R_j(\phi)$, for
$\phi\in L^p(\R^d)$.

We should prove that the H-distribution corresponding to
(the chosen subsequences of) $(u_n)$ and $(v_n)$ satisfies \eqref{pro_7}. To
this end, take the following sequence of test functions:
$$
\phi_n := \varphi_1 (I_1\!\circ\! {\cal A}_{\psi(\xi/|\xi|)})(\ph_2 v_n),
$$
where $\ph_1,\ph_2\in C^\infty_c(\R^d)$ and $\psi\in C^{\kappa}(S^{d-1})$, $\kappa=[d/2]+1$.
Then, apply the right-hand side of \eqref{ns31}, which converges
strongly to $0$ in $W^{-1,q}(\R^d)$ by the assumption, to a weakly
converging sequence $(\phi_n)$ in the dual space $W^{1,q'}(\R^d)$.

We can do that since $(\phi_n)$ is a bounded sequence in $W^{1,r}(\R^d)$ for any $r\in\oi1\infty$.

Indeed, ${\cal A}_{\psi}(\ph_2 v_n)$ is bounded in any $L^r(\R^d)$
($r>1$). By the well known fact \cite[Theorem V.1]{Ste} that $I_1$
is bounded from $L^q(\R^d)$ to $L^{q^\ast}(\R^d)$, for $q\in\oi1d$
and ${1\over q^*}={1\over q}-{1\over d}$, $\phi_n$ is bounded in
$L^{q^\ast}(\R^d)$ for all sufficiently large $q^\ast$. Then, take
$q^\ast\ge r$ and due to the compact support of $\ph_1$ we have that
$L^{q^\ast}$ boundedness implies the same in $L^r$. On the other
hand, $R_j$ is bounded from $L^r(\R^d)$ to itself, for any
$r\in\oi1\infty$, thus $\partial_{x_j} (\varphi_1 (I_1\!\circ
\!{\cal A}_{\psi(\xi/|\xi|)})(\ph_2 v_n))$ is bounded in
$L^r(\R^d)$.

Therefore we have (the sequence is bounded and $0$ is the only accumulation point,
so the whole sequence converges to $0$)
\begin{equation}
\label{zg_2} \lim\limits_{n\to \infty}
\dup{W^{-1,q}(\R^d)}{f_n}{\phi_n}{W^{1,q'}(\R^d)} = 0 \;.
\end{equation}

Concerning the left-hand side of \eqref{ns31},  according to \eqref{V23}
one has
\begin{align}
\label{zg_1}
\dup{W^{-1,q}(\R^d)}{\sum_{j=1}^d
    \partial_{x_j}(A_j u_n)}{\phi_n}{W^{1,q'}(\R^d)}
& = \int_{\R^d}\!\sum_{j=1}^d
    \overline\ph_1 A_j u_n \overline{{\cal A}_{\frac{\xi_j}{|\xi|}\psi(\xi/|\xi|)}(\ph_2u_n)}dx\\
& \!\!\!\!\!- \int_{\R^d}\!\partial_{x_j}\overline\ph_1 \sum_{j=1}^d A_j u_n
    \overline{(I_1\!\circ \!{\cal A}_{\psi(\xi/|\xi|)})(\ph_2 v_n)} dx. \nonumber
\end{align}
The first term on the right is of the form of the right-hand side of
\eqref{basic111}. The integrand in the second term is supported in
a fixed compact and weakly converging to $0$ in $L^p$, so strongly
in $W^{-1,r'}$, where $r$ is
such that $p=r^\ast$ (i.e.~$r=dp/(d-p)$). Of course, the argument
giving the boundedness of $\phi_n$ in $W^{1,q'}(\R^d)$ above applies
also to $r$ instead of $q'$.

Therefore, from \eqref{zg_2} and \eqref{zg_1} we conclude
\eqref{pro_7}.
\end{proof}

\begin{remark}
Notice that the assumption of the strong convergence of $f_n$ in
$W^{-1,q}(\R^d)$ can be relaxed to local convergence, as in the
proof we used a cutoff function $\ph_1$.
\end{remark}

We conclude the paper by another corollary of Theorem \ref{tbasic2}:
the well known Murat-Tartar div-curl lemma in the
$(L^p,L^{p'})$-setting \cite{mur1, tar_cc}.

\begin{theorem}
\label{cc}  Let ${\bf u}_n=(u^1_n,u_n^2)$ and ${\bf
v}_n=(v^1_n,v_n^2)$ be vector valued sequences converging to zero
weakly in $L^p(\R^2)$ and $L^{p'}(\R^2)$, respectively.

Assume the sequence $({\rm div} \,{\bf u}_n)=(\pa_x u^1_n+\pa_y
u^2_n)$ is bounded in $L^p(\R^2)$, and the sequence $({\rm curl}
\,{\bf v}_n)=(\pa_y v^1_n-\pa_x v^2_n)$ is bounded in
$L^{p'}(\R^2)$.

Then, the sequence $(u^1_n v^1_n + u^2_n v^2_n)$ converges to zero in the sense of
distributions (or vaguely in the sense of Radon measures).
\end{theorem}

\begin{proof}
Denote by $\mu^{ij}$ the $H$-distribution corresponding to (some sub)
sequences (of) $(u^i_n)$ and $(v^j_n)$, $i,j=1,2$.

Since $(\pa_x u^1_n+\pa_y u^2_n)$ is bounded in $L^p(\R^2)$, and
$(\pa_y v^1_n-\pa_x v_n^2)$ is bounded in $L^{p'}(\R^2)$, they are weakly precompact,
while the only possible limit is zero, so
\begin{equation}
\label{novo-3}
\begin{split}
& \pa_x u^1_n + \pa_y
u^2_n\rightharpoonup 0 \ \ {\rm in} \ \ L^p\;,\qquad\mbox{and}\\
& \pa_y v^1_n-\pa_x v_n^2 \rightharpoonup 0 \ \ {\rm in} \ \ L^{p'}.
\end{split}
\end{equation}
Now, from the compactness properties of the Riesz potential $I_1$
(see the proof of previous theorem), we conclude that for every
$\varphi\in C_c(\R^2)$ the following limit holds strongly in
$L^p(\R^2)$:
\begin{align}
\label{brg_1} &{\cal A}_{\psi(\xi/|\xi|)\frac{\xi_1}{|\xi|}}(\varphi
u_n^1)+{\cal A}_{\psi(\xi/|\xi|)\frac{\xi_2}{|\xi|}}(\varphi
u_n^2)={\cal A}_{\frac{\psi(\xi/|\xi|)}{|\xi|}}(\pa_x(\varphi
u_n^1)+\pa_y(\varphi u_n^2))\to 0\;.
\end{align}
Multiplying \eqref{brg_1} first by $\varphi
v_n^1$ and then by $\varphi v_n^2$, integrating over $\R^2$ and
passing to the limit $n\to \infty$, we conclude from \eqref{basic111}, due to
the arbitrariness of $\psi$ and $\varphi$:
\begin{equation}
\label{brg_3} \xi_1\mu^{11}+\xi_2\mu^{21}=0,\qquad \mbox{and}\qquad
\xi_1\mu^{12}+\xi_2\mu^{22}=0\,.
\end{equation}

Next, take
$$w^j_n=\varphi {\cal A}_{\frac{\psi(\xi/|\xi|)}{|\xi|}}(\varphi u_n^j) \in
W^{1,p'}(\R^d), \ \ j=1,2.
$$

From \eqref{novo-3} we get
\begin{align*}
&\langle (\varphi v_n^1,-\varphi v_n^2),\nabla w^j_n
\rangle=-\langle {\rm curl} (\varphi v_n^1,\varphi v_n^2), w^j_n
\rangle \to 0 \ \ {\rm as} \ \ n\to \infty,
\end{align*}
for $j=1,2$. Rewriting it in the integral formulation, we obtain
from \eqref{basic111} in the same way that we obtained
\eqref{brg_3}:
\begin{equation}
\label{brg_4} \xi_2\mu^{11}-\xi_1\mu^{12}=0,\ \
\xi_2\mu^{21}-\xi_1\mu^{22}=0.
\end{equation}
From the algebraic relations \eqref{brg_3} and \eqref{brg_4}, we can easily conclude
\begin{align*}
&\xi_1\left(\mu^{11}+\mu^{22} \right)=0 \ \ {\rm and} \ \
\xi_2\left(\mu^{11}+\mu^{22} \right)=0,
\end{align*}
implying that the measure $\mu^{11}+\mu^{22}$ is
supported on the set $\{\xi_1=0\}\cap \{\xi_2=0\}\cap P=\emptyset$,
which implies $\mu^{11}+\mu^{22}\equiv 0$.

After inserting $\psi\equiv 1$ in the definition of $H$-distribution
\eqref{basic111}, we immediately reach the conclusion.
\end{proof}

\noindent{\bf Acknowledgement:} Originally, Theorem \ref{tbasic2} was
proved only in the case $q=\infty$. We would like to thank
Martin Lazar for pointing out the possibility to extend the
theorem to more general values of $q$.

\end{document}